\documentclass[11pt]{article}
\usepackage{amsmath,amssymb,amsthm}
\usepackage{epsfig}
\usepackage{color}
\usepackage{enumerate}
\usepackage{algorithm}
\usepackage{algorithmic}
\usepackage{hyperref}
\usepackage{enumerate}
\usepackage{graphicx}
\usepackage{tikz}
\usetikzlibrary{arrows.meta,shapes.arrows}
\hypersetup{colorlinks=true, linkcolor=blue, citecolor=blue,
urlcolor=blue}
\usetikzlibrary{calc}

\hypersetup{colorlinks=true}

\hypersetup{colorlinks=true, linkcolor=blue, citecolor=blue,urlcolor=blue}

\title{Further results on outer independent $2$-rainbow dominating functions of graphs\thanks{This research has been supported by the Discrete Mathematics Laboratory of the Faculty of Mathematical Sciences at Alzahra University.}}
\date{}
\author{{Babak Samadi\thanks{Corresponding author} and Nasrin Soltankhah}\vspace{1mm}\\
{Department of Mathematics, Faculty of Mathematical Sciences, Alzahra University,}\\ {Tehran, Iran}\vspace{1mm}\\
{b.samadi@alzahra.ac.ir}\\
{soltan@alzahra.ac.ir}}
\date{}

\newtheorem{theorem}{Theorem}[section]

\theoremstyle{definition}

\theoremstyle{remark}

\begin{document}

\maketitle

\begin{abstract}
Let $G=(V(G),E(G))$ be a graph. A function $f:V(G)\rightarrow \mathbb{P}(\{1,2\})$ is a $2$-rainbow dominating function if for every vertex $v$ with $f(v)=\emptyset$, $f\big{(}N(v)\big{)}=\{1,2\}$. An outer-independent $2$-rainbow dominating function (OI$2$RD function) of $G$ is a $2$-rainbow dominating function $f$ for which the set of all $v\in V(G)$ with $f(v)=\emptyset$ is independent. The outer independent $2$-rainbow domination number (OI$2$RD number) $\gamma_{oir2}(G)$ is the minimum weight of an OI$2$RD function of $G$.

In this paper, we first prove that $n/2$ is a lower bound on the OI$2$RD number of a connected claw-free graph of order $n$ and characterize all such graphs for which the equality holds, solving an open problem given in an earlier paper. In addition, a study of this parameter for some graph products is carried out. In particular, we give a closed (resp. an exact) formula for the OI$2$RD number of rooted (resp. corona) product graphs and prove upper bounds on this parameter for the Cartesian product and direct product of two graphs. 
\end{abstract}


\section{Introduction and preliminaries} 

Throughout this paper, we consider $G$ as a finite simple graph with vertex set $V(G)$ and edge set $E(G)$. We use \cite{West} as a reference for terminology and notation which are not explicitly defined here. The {\em open neighborhood} of a vertex $v$ is denoted by $N(v)$, and its {\em closed neighborhood} is $N[v]=N(v)\cup \{v\}$. The {\em minimum} and {\em maximum degrees} of $G$ are denoted by $\delta(G)$ and $\Delta(G)$, respectively.

Given a graph $G$, a subset $S\subseteq V(G)$ is said to be a {\em dominating set} in $G$ if every vertex not in $S$ is adjacent to a vertex in $S$. The {\em domination number} $\gamma(G)$ is the minimum cardinality of a dominating set in $G$. 

Domination presents a model for a situation in which every empty location (vertex with no guards) needs to be protected by a guard occupying a neighboring location. A generalization of domination was proposed in \cite{Bresar}, where different types of guards are deployed, and the empty locations must have all types of guards in their neighborhoods. This relaxation led to the definition of $k$-rainbow domination. Indeed, a function $f:V(G)\rightarrow \mathbb{P}(\{1,\cdots,k\})$ is a \textit{$k$-rainbow dominating function} ($k$RD function) if for every vertex $v$ with $f(v)=\emptyset$, $f(N(v))=\{1,\cdots,k\}$. The \textit{$k$-rainbow domination number} $\gamma_{rk}(G)$ is the minimum weight of $\sum_{v\in V(G)}|f(v)|$ taken over all $k$RD functions of $G$. This concept was formally defined by Bre\v{s}ar et al. \cite{Bresar}. 

The existence of two adjacent locations with no guards can jeopardize them. Indeed, they would be considered more vulnerable. One improved situation for a location with no guards is to be surrounded by locations in which guards are stationed. This motivates us to consider a $kRD$ function $f$ for which the set of vertices assigned $\emptyset$ under $f$ is independent. More formally, we have the following definition. A function $f$ is an \textit{outer independent $k$-rainbow dominating function} (OI$k$RD function) of $G$ if $f$ is a $k$RD function and the set of vertices with weight $\emptyset$ is an independent set. The \textit{outer independent $k$-rainbow domination number} (OI$k$RD number) $\gamma_{oirk}(G)$ is the minimum weight of an OI$k$RD function of $G$. An OI$k$RD function of weight $\gamma_{oirk}(G)$ is called a $\gamma_{oirk}(G)$-function. This concept was first introduced by Kang et al. \cite{qssss} and studied in \cite{acm,mm}. Mansouri and Mojdeh \cite{mm} showed that the problem of computing the OI$2$RD number is NP-hard even when restricted to planar graphs with maximum degree at most four and triangle-free graphs. 

In this paper, emphasizing the case $k=2$, we first provide a characterization of all connected claw-free graphs whose OI$2$RD numbers are equal to half of their orders, solving an open problem from \cite{mm}. In the second section of the paper, we investigate the OI$2$RD numbers of some graph products such as the Cartesian, direct, rooted and corona products of graphs. We refer the readers to the book \cite{ImKl} for a comprehensive survey of the graph products.

For any function $f:V(G)\rightarrow \mathbb{P}(\{1,2\})$, we let $V_{\emptyset}$, $V_{\{1\}}$, $V_{\{2\}}$ and $V_{\{1,2\}}$ stand for the set of vertices assigned with $\emptyset$, $\{1\}$, $\{2\}$ and $\{1,2\}$ under $f$, respectively. Since these four sets determine $f$, we can equivalently write $f=(V_{\emptyset},V_{\{1\}},V_{\{2\}},V_{\{1,2\}})$. Note that $w(f)=|V_{\{1\}}|+|V_{\{2\}}|+2|V_{\{1,2\}}|$ is the weight of $f$.    


\section{Claw-free graphs}

Mansouri and Mojdeh \cite{mm} proved that the OI$2$RD number of a $K_{1,r}$-free graph $G$ of order $n$ with $s'$ strong support vertices can be bounded from below by $2(n+s')/(1+r)$. They also posed the open problem of characterizing all $K_{1,r}$-free (or at least claw-free) graphs for which the lower bound holds with equality. Our aim in this section is to solve the problem for the claw-free graphs (that is, the case $r=3$). Let $G$ be a claw-free graph with components $G_{1},\cdots,G_{t}$. Since $\gamma_{oir2}(G)=\sum_{i=1}^{t}\gamma_{oir2}(G_{i})$, it follows that $\gamma_{oir2}(G)=(n+s')/2$ if and only if $\gamma_{oir2}(G_{i})=(|V(G_{i})|+s_{i}')/2$ for each $1\leq i\leq t$, in which $s_{i}'$ is the number of strong support vertices of $G_{i}$. In such a case, $s'=\sum_{i=1}^{t}s_{i}'=0$, unless $G_{i}=P_{3}$ for some $1\leq i\leq t$. For such a component, we have $\gamma_{oir2}(G_{i})=(|V(G_{i})|+s_{i}')/2$. So, in what follows, we may assume that $G$ is connected and $s'=0$. In order to solve the problem in such a case, it suffices to characterize all connected claw-free graphs $G$ of order $n$ for which the equality holds in the lower bound $n/2$ on $\gamma_{oir2}(G)$.

In this section, we show that the OI$2$RD number of a claw-free graph can be bounded from below by half of its order. In order to characterize all claw-free graphs attaining this bound, we call a graph of the following form a \textit{$k$-unit} in which the number of triangles is $k-1$.
\begin{figure}[h]\vspace{-25mm}
\tikzstyle{every node}=[circle, draw, fill=white!, inner sep=0pt,minimum width=.16cm]
\begin{center}
\begin{tikzpicture}[thick,scale=.6]
  \draw(0,0) { 

+(-5,-5) node{}
+(-4,-5) node{}
+(-3,-5) node{}
+(-1.75,-5) node{}
+(-0.75,-5) node{}

+(-4.5,-6) node{}
+(-3.5,-6) node{}
+(-1.25,-6) node{}

+(-5,-5) -- +(-4,-5) -- +(-3,-5)
+(-1.75,-5) -- +(-0.75,-5)
+(-5,-5) -- +(-4.5,-6) -- +(-4,-5) -- +(-3.5,-6) -- +(-3,-5)
+(-1.75,-5) -- +(-1.25,-6) -- +(-0.75,-5)

+(-2.75,-5.2) node[rectangle, draw=white!0, fill=white!100]{${\textbf{.}}$}
+(-2.5,-5.2) node[rectangle, draw=white!0, fill=white!100]{${\textbf{.}}$}
+(-2.25,-5.2) node[rectangle, draw=white!0, fill=white!100]{${\textbf{.}}$}
+(-2,-5.2) node[rectangle, draw=white!0, fill=white!100]{${\textbf{.}}$}

+(-5.2,-4.6) node[rectangle, draw=white!0, fill=white!100]{${\tiny v_{1}}$}
+(-4,-4.6) node[rectangle, draw=white!0, fill=white!100]{${\tiny v_{2}}$}
+(-3,-4.6) node[rectangle, draw=white!0, fill=white!100]{${\tiny v_{3}}$}
+(-1.75,-4.6) node[rectangle, draw=white!0, fill=white!100]{${\tiny v_{k-1}}$}
+(-0.45,-4.6) node[rectangle, draw=white!0, fill=white!100]{${\tiny v_{k}}$}

};
\end{tikzpicture}
\end{center}
\caption{A $k$-unit.}\label{fig1}
\end{figure}
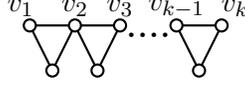\\
Note that a $1$-unit is isomorphic to $K_{1}$. We now let $\mathcal{G}$ be the family of all graphs of the form $G_{1}$, $G_{2}$ and $G_{3}$ depicted in Figure \ref{fig2}.

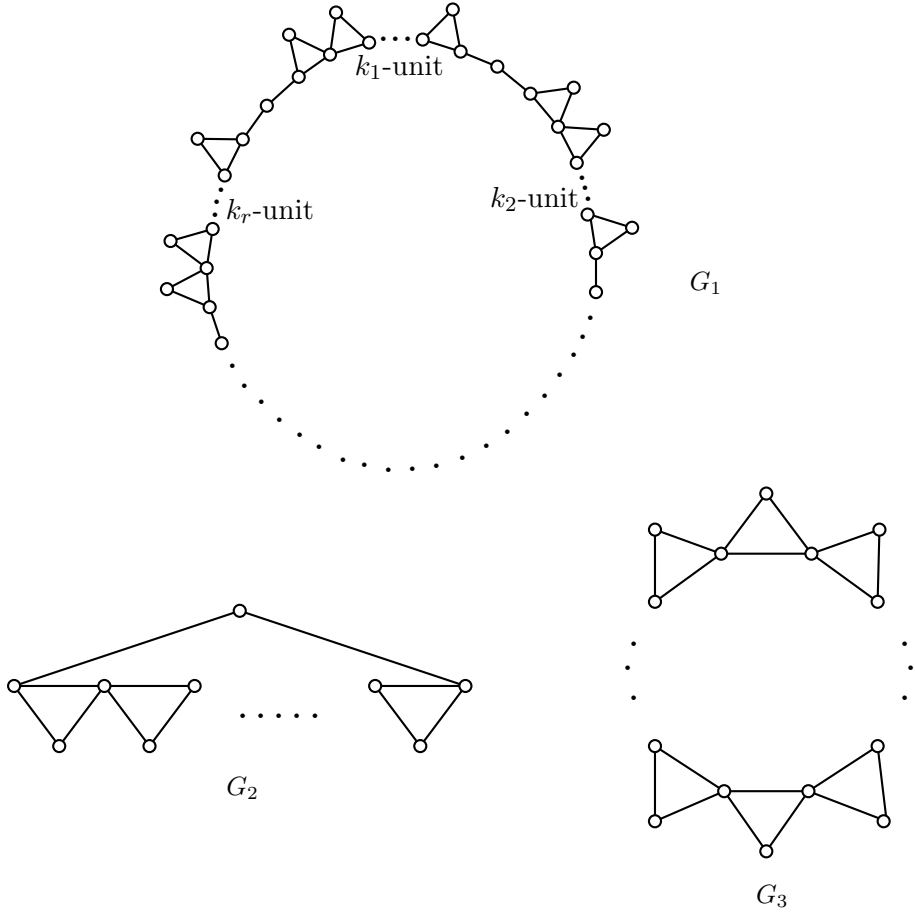
\begin{figure}[h]
\tikzstyle{every node}=[circle, draw, fill=white!, inner sep=0pt,minimum width=.16cm]
\begin{center}
\begin{tikzpicture}[thick,scale=0.8]
  \draw(0,0) { 
  
+(-11.8,9.15) node{}
+(-11.27,9.63) -- +(-11.8,9.15)

+(-12.95,8.6) node{}
+(-12.2,8.59) -- +(-12.95,8.6) -- +(-12.5,7.99)

+(-12.2,8.59) node{}   
+(-11.8,9.15) -- +(-12.2,8.59) 

+(-12.5,7.99) node{} 
+(-12.2,8.59) -- +(-12.5,7.99)

+(-12.56,7.75) node[rectangle, draw=white!0, fill=white!100]{ ${\textbf{.}}$} 
+(-12.63,7.56) node[rectangle, draw=white!0, fill=white!100]{ ${\textbf{.}}$}
+(-12.66,7.34) node[rectangle, draw=white!0, fill=white!100]{ ${\textbf{.}}$}
+(-11.75,7.4) node[rectangle, draw=white!0, fill=white!100]{$k_{r}$-unit} 

+(-12.7,7.1) node{}
+(-12.8,6.45) node{}

+(-12.7,7.1) -- +(-12.8,6.45)

+(-13.4,6.9) node{}
+(-12.7,7.1) -- +(-13.4,6.9) -- +(-12.8,6.45)

+(-12.75,5.8) node{}
+(-12.8,6.45) -- +(-12.75,5.8)

+(-13.46,6.1) node{}
+(-12.8,6.45) -- +(-13.46,6.1) -- +(-12.75,5.8)

+(-12.55,5.2) node{}
+(-12.55,5.2) -- +(-12.75,5.8)

+(-10.75,10) node{}
+(-10.1,10.2) node{}
+(-11.27,9.63) node{}

+(-9.86,10.26) node[rectangle, draw=white!0, fill=white!100]{ ${\textbf{.}}$}
+(-9.66,10.265) node[rectangle, draw=white!0, fill=white!100]{ ${\textbf{.}}$}
+(-9.46,10.26) node[rectangle, draw=white!0, fill=white!100]{ ${\textbf{.}}$}
+(-9.6,9.8) node[rectangle, draw=white!0, fill=white!100]{$k_{1}$-unit}

+(-11.27,9.63) -- +(-11.43,10.33) -- +(-10.75,10) -- +(-10.65,10.7) -- +(-10.1,10.2)
+(-9.21,10.25) -- +(-8.71,10.75) -- +(-8.58,10.05)

+(-9.21,10.25) node{}
+(-8.58,10.05) node{}

+(-11.43,10.33) node{}
+(-10.65,10.7) node{}
+(-8.71,10.75) node{}

+(-11.27,9.63) -- +(-10.75,10) -- +(-10.1,10.2)
+(-9.21,10.25) -- +(-8.58,10.05)

+(-7.97,9.8) node{}

+(-7.42,9.35) node{}
+(-6.96,8.8) node{}
+(-6.65,8.2) node{}

+(-6.56,7.96) node[rectangle, draw=white!0, fill=white!100]{ ${\textbf{.}}$}
+(-6.51,7.79) node[rectangle, draw=white!0, fill=white!100]{ ${\textbf{.}}$}
+(-6.46,7.6) node[rectangle, draw=white!0, fill=white!100]{ ${\textbf{.}}$}
+(-7.35,7.6) node[rectangle, draw=white!0, fill=white!100]{$k_{2}$-unit}

+(-6.47,7.34) node{}
+(-6.33,6.7) node{}

+(-8.58,10.05) -- +(-7.97,9.8) -- +(-7.42,9.35) -- +(-6.96,8.8) -- +(-6.65,8.2)
+(-6.43,7.34) -- +(-6.33,6.7)

+(-6.7,9.45) node{}
+(-7.42,9.35) -- +(-6.7,9.45) -- +(-6.96,8.8)

+(-6.2,8.75) node{}
+(-6.96,8.8) -- +(-6.2,8.75) -- +(-6.65,8.2)

+(-5.73,7.12) node{}
+(-6.43,7.34) -- +(-5.73,7.12) -- +(-6.33,6.7)

+(-6.33,6.05) node{}
+(-6.33,6.7) -- +(-6.33,6.05)

+(-6.42,5.62) node[rectangle, draw=white!0, fill=white!100]{ ${\textbf{.}}$}
+(-6.55,5.3) node[rectangle, draw=white!0, fill=white!100]{ ${\textbf{.}}$}
+(-6.7,5.01) node[rectangle, draw=white!0, fill=white!100]{ ${\textbf{.}}$}
+(-6.93,4.68) node[rectangle, draw=white!0, fill=white!100]{ ${\textbf{.}}$}
+(-7.18,4.32) node[rectangle, draw=white!0, fill=white!100]{ ${\textbf{.}}$}
+(-7.48,4.03) node[rectangle, draw=white!0, fill=white!100]{ ${\textbf{.}}$}
+(-7.8,3.72) node[rectangle, draw=white!0, fill=white!100]{ ${\textbf{.}}$}
+(-8.15,3.49) node[rectangle, draw=white!0, fill=white!100]{ ${\textbf{.}}$}
+(-8.55,3.3) node[rectangle, draw=white!0, fill=white!100]{ ${\textbf{.}}$}
+(-8.98,3.17) node[rectangle, draw=white!0, fill=white!100]{ ${\textbf{.}}$}
+(-9.38,3.11) node[rectangle, draw=white!0, fill=white!100]{ ${\textbf{.}}$}
+(-9.78,3.1) node[rectangle, draw=white!0, fill=white!100]{ ${\textbf{.}}$}
+(-10.18,3.18) node[rectangle, draw=white!0, fill=white!100]{ ${\textbf{.}}$}
+(-10.53,3.26) node[rectangle, draw=white!0, fill=white!100]{ ${\textbf{.}}$}
+(-10.93,3.48) node[rectangle, draw=white!0, fill=white!100]{ ${\textbf{.}}$}
+(-11.25,3.67) node[rectangle, draw=white!0, fill=white!100]{ ${\textbf{.}}$}
+(-11.59,3.93) node[rectangle, draw=white!0, fill=white!100]{ ${\textbf{.}}$}
+(-11.89,4.22) node[rectangle, draw=white!0, fill=white!100]{ ${\textbf{.}}$}
+(-12.17,4.5) node[rectangle, draw=white!0, fill=white!100]{ ${\textbf{.}}$}
+(-12.37,4.83) node[rectangle, draw=white!0, fill=white!100]{ ${\textbf{.}}$}
  

+(-16,-0.5) node{}
+(-14.5,-0.5) node{}
+(-13,-0.5) node{}
+(-10,-0.5) node{}
+(-8.5,-0.5) node{}
+(-16,-0.5) -- +(-14.5,-0.5) -- +(-13,-0.5)
+(-10,-0.5) -- +(-8.5,-0.5)

+(-15.25,-1.5) node{}
+(-13.75,-1.5) node{}
+(-9.25,-1.5) node{}
+(-16,-0.5) -- +(-15.25,-1.5) -- +(-14.5,-0.5) -- +(-13.75,-1.5) -- +(-13,-0.5)
+(-10,-0.5) -- +(-9.25,-1.5) -- +(-8.5,-0.5)

+(-12.25,0.75) node{}
+(-16,-0.5) -- +(-12.25,0.75) -- +(-8.5,-0.5)

+(-12.2,-1) node[rectangle, draw=white!0, fill=white!100]{ ${\textbf{.}}$}
+(-11.9,-1) node[rectangle, draw=white!0, fill=white!100]{ ${\textbf{.}}$}
+(-11.6,-1) node[rectangle, draw=white!0, fill=white!100]{ ${\textbf{.}}$}
+(-11.3,-1) node[rectangle, draw=white!0, fill=white!100]{ ${\textbf{.}}$}
+(-11,-1) node[rectangle, draw=white!0, fill=white!100]{ ${\textbf{.}}$}

                             
+(-4.25,1.7) node{}
+(-2.75,1.7) node{}
+(-1.65,0.9) node{}
+(-5.35,0.9) node{}

+(-1.2,0.2) node[rectangle, draw=white!0, fill=white!100]{ ${\textbf{.}}$}
+(-1.1,-0.2) node[rectangle, draw=white!0, fill=white!100]{ ${\textbf{.}}$}
+(-1.2,-0.7) node[rectangle, draw=white!0, fill=white!100]{ ${\textbf{.}}$}

+(-5.7,0.2) node[rectangle, draw=white!0, fill=white!100]{ ${\textbf{.}}$}
+(-5.8,-0.2) node[rectangle, draw=white!0, fill=white!100]{ ${\textbf{.}}$}
+(-5.7,-0.7) node[rectangle, draw=white!0, fill=white!100]{ ${\textbf{.}}$}

+(-4.2,-2.25) node{}
+(-2.8,-2.25) node{}
+(-1.65,-1.5) node{}
+(-5.35,-1.5) node{}

+(-5.35,0.9) -- +(-4.25,1.7) -- +(-2.75,1.7) -- +(-1.65,0.9) 
+(-5.35,-1.5) -- +(-4.2,-2.25) -- +(-2.8,-2.25) -- +(-1.65,-1.5)

+(-3.5,2.7) node{}
+(-1.62,2.1) node{}
+(-5.35,2.1) node{}

+(-3.5,-3.25) node{}
+(-1.55,-2.75) node{}
+(-5.35,-2.75) node{}

+(-5.35,0.9) -- +(-5.35,2.1) -- +(-4.25,1.7) -- +(-3.5,2.7) -- +(-2.75,1.7) -- +(-1.62,2.1) -- +(-1.65,0.9) 
+(-5.35,-1.5) -- +(-5.35,-2.75) -- +(-4.2,-2.25) -- +(-3.5,-3.25) -- +(-2.8,-2.25) -- +(-1.5,-2.75) -- +(-1.65,-1.5)


+((-4.5,6.2) node[rectangle, draw=white!0, fill=white!100]{{\small $G_1$}}
+((-12.2,-2.2) node[rectangle, draw=white!0, fill=white!100]{{\small $G_2$}}
+((-3.4,-4) node[rectangle, draw=white!0, fill=white!100]{{\small $G_3$}}
};
\end{tikzpicture}
\end{center}
\caption{The claw-free graphs $G_{1}$, $G_{2}$ and $G_{3}$. In $G_{2}$ (resp. $G_{3}$), the number of triangles is odd (resp. even). In $G_{1}$, $k_{1}+\cdots+k_{r}$ is even.}\label{fig2}
\end{figure} 

\begin{theorem}\label{claw-free}
Let $G$ be a connected claw-free graph of order $n$. Then, $\gamma_{oir2}(G)\geq n/2$ with equality if and only if $G\in \mathcal{G}$.
\end{theorem}
\begin{proof}
Let $f$ be a $\gamma_{oir2}(G)$-function. We set $Q=V_{\emptyset}\cap N(V_{\{1,2\}})$. Since $G$ is a claw-free graph and because $V_{\emptyset}$ is independent, every vertex in $V_{\{1,2\}}$ has at most two neighbors in $Q$. Thus, $|Q|\leq2|V_{\{1,2\}}|$. On the other hand, every vertex in $V_{\emptyset}\setminus Q$ has at least two neighbors in $V_{\{1\}}\cup V_{\{2\}}$. This implies that $2|V_{\emptyset}\setminus Q|\leq|[V_{\emptyset}\setminus Q,V_{\{1\}}\cup V_{\{2\}}]|\leq2(|V_{\{1\}}|+|V_{\{2\}}|)$. So, $|V_{\emptyset}\setminus Q|\leq|V_{\{1\}}|+|V_{\{2\}}|$. We now have 
\begin{equation}\label{claw-free1}
\begin{array}{lcl}
2\big{(}n-\gamma_{oir2}(G)\big{)}&\leq&2(n-|V_{\{1\}}|-|V_{\{2\}}|-|V_{\{1,2\}}|)=2|V_{\emptyset}|=2|Q|+2|V_{\emptyset}\setminus Q|\\
&\leq&2(|V_{\{1\}}|+|V_{\{2\}}|+2|V_{\{1,2\}}|)=2\gamma_{oir2}(G),
\end{array}
\end{equation}
implying the lower bound.

Suppose that the equality holds for a connected claw-free graph $G$. Then, all inequalities in (\ref{claw-free1}) necessarily hold with equality. In particular, $V_{\{1,2\}}=\emptyset$ (and consequently $Q=\emptyset$) by the equality instead of the first inequality in (\ref{claw-free1}). This implies that every vertex in $V_{\emptyset}$ has at least one neighbor in each of $V_{\{1\}}$ and $V_{\{2\}}$. Taking this fact into account, the resulting equality $2|V_{\emptyset}|=|[V_{\emptyset},V_{\{1\}}\cup V_{\{2\}}]|$ shows that every vertex in $V_{\emptyset}$ has precisely one neighbor in each of $V_{\{1\}}$ and $V_{\{2\}}$. On the other hand, $|[V_{\emptyset},V_{\{1\}}\cup V_{\{2\}}]|=2(|V_{\{1\}}|+|V_{\{2\}}|)$ follows that every vertex in $V_{\{1\}}\cup V_{\{2\}}$ is adjacent to exactly two vertices in $V_{\emptyset}$.

Let $H=G[V_{\{1\}}\cup V_{\{2\}}]$. Suppose to the contrary that $\deg_{H}(v)\geq3$ for some $v\in V(H)$. Since $G$ is claw-free and because $v$ has two neighbors in $V_{\emptyset}$, it follows that every vertex in $N_{H}(v)$ must be adjacent to at least one of the two neighbors of $v$, say $x_{1}$ and $x_{2}$, in $V_{\emptyset}$. This implies that $\deg(x_{1})\geq3$ or $\deg(x_{2})\geq3$, a contradiction. The above discussion guarantees that $\Delta(H)\leq2$, and thus $H$ is isomorphic to a disjoint union of some cycles and paths. 

Let $H'$ be a cycle $v_{1}v_{2}\cdots v_{t}v_{1}$ as a component of $H$. Let $v_{11}$ and $v_{12}$ be the neighbors of $v_{1}$ in $V_{\emptyset}$. Since $G$ is claw-free, both $v_{2}$ and $v_{t}$ have at least one neighbor in $\{v_{11},v_{12}\}$. Furthermore, because $\deg(v_{11})=\deg(v_{12})=2$, both $v_{2}$ and $v_{t}$ have exactly one neighbor in $\{v_{11},v_{12}\}$. We may assume, without loss of generality, that $v_{11}v_{t},v_{12}v_{2}\in E(G)$. Let $v_{22}$ be the second neighbor of $v_{2}$ in $V_{\emptyset}$. Again, because $G$ is claw-free and $v_{12}v_{22},v_{12}v_{3}\notin E(G)$, we infer that $v_{22}v_{3}\in E(G)$. Iterating this process results in a graph of the form $G_{3}$, in which $v_{1}v_{2}\cdots v_{t}v_{1}$ is the resulting cycle by removing all vertices of degree two. In addition, since all the vertices in $V_{\emptyset}$ have degree two, both $\{1\}$ and $\{2\}$ must appear on their neighbors on the cycle. This implies that $t$ is even. Note that each vertex of $H'$ has no other neighbors in $G$. Now $G$ is of the form $G_{3}$ by its connectedness.

In what follows, we may assume that $H$ does not have any cycle as a component. If $H$ is an edgeless graph, then $G$ is isomorphic to the cycle $C_{n}$. Notice that
 \begin{equation}\label{cycle}
\gamma_{oir2}(C_{p})=\lfloor p/2\rfloor+\lceil p/4\rceil-\lfloor p/4\rfloor
\end{equation}
for $p\geq3$ (see \cite{mm}). Since $\gamma_{oir2}(C_{n})=n/2$, the formula (\ref{cycle}) shows that $n\equiv0$ (mod $4$). It is then easy to observe that $G$ is of the form of $G_{1}$, in which $r=n/2$ and $k_{1}=\cdots=k_{r}=1$. Suppose now that $H$ is not edgeless and let $H''$ be a path $v_{1}v_{2}\cdots v_{t}$, on $t\geq2$ vertices, as a component of $H$. Let $v_{2}$ be adjacent to two vertices $v_{21}$ and $v_{22}$ in $V_{\emptyset}$. Note that $v_{1}$ must be adjacent to at least one of $v_{21}$ and $v_{22}$, for otherwise $G$ would have a claw as an induced subgraph. If $v_{1}$ is adjacent to both $v_{21}$ and $v_{22}$, then $G[v_{1},v_{2},v_{21}]$ is a $2$-unit and $G\cong G[v_{1},v_{2},v_{21},v_{22}]\cong K_{4}-v_{21}v_{22}$. In such a case, $G$ is of the form $G_{2}$ since it is connected. So, $G\in \mathcal{G}$. In what follows, we assume that $v_{1}$ is adjacent to only one of $v_{21}$ and $v_{22}$. We then proceed with $v_{2}$. Since $G$ is claw-free, it follows that both $v_{1}$ and $v_{3}$ must have neighbors in $\{v_{21},v_{22}\}$. Moreover, $\deg(v_{21})=\deg(v_{22})=2$ shows that both $v_{1}$ and $v_{3}$ have exactly one neighbor in $\{v_{21},v_{22}\}$. So, we may assume that $v_{1}v_{21},v_{3}v_{22}\in E(G)$. Similarly, $v_{3}$ has two neighbors $v_{31}$ and $v_{32}$ in $V_{\emptyset}$, in which we may assume that $v_{31}=v_{22}$. By repeating this process we obtain a $t$-unit on the set of vertices $K=\{v_{1},v_{2},\cdots,v_{t},u_{12},u_{23},\cdots,u_{(t-1)t}\}$ in which $u_{i(i+1)}$ is adjacent to both $v_{i}$ and $v_{i+1}$, for $1\leq i\leq t-1$. We now consider two cases depending on $A=\big{(}N(v_{1})\cap N(v_{t})\big{)}\setminus\{v_{2}\}$.

\textit{Case $1$.} $A\neq\emptyset$. Let $u_{1t}$ be in $A$. Because the path $H''$ is a component of $H$, $A\cap V(H)=\emptyset$. Therefore, $u_{1t}\in V_{\emptyset}$. Since both $v_{1}$ and $v_{t}$ have exactly two neighbors in $V_{\emptyset}$, it follows that $A=\{u_{1t}\}$. In such a situation, the subgraph induced by $K\cup\{u_{1t}\}$ is isomorphic to $G$ because it is connected. Notice that since all vertices in $V_{\emptyset}$ have degree two, both $\{1\}$ and $\{2\}$ must appear on their neighbors in the $\{v_{1},\cdots,v_{t}\}$. This shows that $t$ is even and hence $G$ is of the form $G_{2}$.

\textit{Case $2$.} $A=\emptyset$. This implies that the subgraph induced by $K$ is a $t$-unit. Then $v_{t}$ is adjacent to a vertex $x'\in V_{\emptyset}$ and $x'$ is adjacent to vertex $x\notin V(H)\setminus V(H'')$ (note that if $x\in V(H'')$, then $x=v_{j}\in\{v_{1},\cdots,v_{t-1}\}$. If $j=1$, then $G$ is of the form $G_{2}$ and hence $G\in \mathcal{G}$. If $j\geq2$, then $v_{j}$ has at least three neighbors in $V_{\emptyset}$ which is impossible). Let $x$ belong to a component $H'''$ of $H$. Since $H$ does not have a cycle as a component, it follows that $H'''$ is a path. In such a case, the vertices of $H'''$ belong to a $|V(H''')|$-unit by a similar fashion. Iterating this process we obtain some $|V(H_{1})|,\cdots,|V(H_{s})|$-units constructed as above, in which $H_{1}=H''$ and $s$ is the largest integer for which there exists such a $|V(H''_{s})|$-unit. Let $H_{s}=w_{1}\cdots w_{p}$ and $w_{p}$ be the vertex which has only one neighbor in $V_{\emptyset}$ in the subgraph induced by $V(H_{1})\cup\cdots\cup V(H_{r})$. Similar to Case $1$, there exists a vertex $u_{1p}\in V_{\emptyset}$ adjacent to both $v_{1}$ and $w_{p}$. On the other hand, both $\{1\}$ and $\{2\}$ must appear on the neighbors of each vertex in $V_{\emptyset}$. This implies that $\sum_{i=1}^{r}|V(H_{i})|$ must be even. Therefore, $G$ is of the form $G_{1}$. 

In both cases above, we have concluded that $G\in \mathcal{G}$.

Conversely, let $G\in \mathcal{G}$. Suppose first that $G$ is of the form $G_{2}$. Let $v_{1}\cdots v_{2t}$ be the path on the set of vertices of degree at lest three of $G_{2}$. Then $\big{(}f(v_{2i-1}),f(v_{2i})\big{)}=(\{1\},\{2\})$ for $1\leq i\leq t$, and $f(v)=\emptyset$ for the other vertices defines an OI$2$RD function with weight half of the order. Let $G$ be of the form $G_{1}$. Let $\{u_{1},\cdots,u_{2p}\}$ be the set of vertices of degree at least three of $G_{1}$ such that $u_{1}\cdots u_{k_{1}}$ is the path in the $k_{1}$-unit, $u_{k_{1}+1}\cdots u_{k_{1}+k_{2}}$ is the path in the $k_{2}$-unit, and so on. It is easy to see that $\big{(}g(u_{2i-1}),g(u_{2i})\big{)}=(\{1\},\{2\})$ for $1\leq i\leq p$, and $g(u)=\emptyset$ for the other vertices is an OI$2$RD function with weight $n/2$. Finally, we suppose that $G$ is of the form $G_{3}$. Let $x_{1}x_{2}\cdots x_{2q}x_{1}$ be the cycle on the vertices of degree four. Then the assignment $\big{(}g(x_{2i-1}),g(x_{2i})\big{)}=(\{1\},\{2\})$ for $1\leq i\leq q$, and $g(x)=\emptyset$ for the other vertices defines an OI$2$RD function of $G_{3}$ with weigh half of its order. Therefore, in all three possibilities, we have concluded that $\gamma_{oir2}(G)=n/2$. This completes the proof. 
\end{proof}


\section{OI$2$RD number of some graph products}

In this section, we consider the OI$2$RD number of direct, Cartesian, rooted and corona products of two graphs.

\subsection{Direct and Cartesian Products}

For the following two standard products of graphs $G$ and $H$ (see \cite{ImKl}), the vertex set of the product is $V(G)\times V(H)$. In the edge set of the \emph{direct product} $G\times H$, two vertices are adjacent if they are adjacent in both coordinates. On the other hand, in the edge set of the \emph{Cartesian product} $G\Box H$, two vertices are adjacent if they are adjacent in one coordinate and equal in the other.  

For a graph $G$, we let $I_{G}$ denote the set of isolated vertices of $G$. By $G^-$ we denote the graph obtained from $G$ by removing all the isolated vertices of $G$. We observe that $I=\big{(}I_{G}\times V(H)\big{)}\cup \big{(}V(G)\times I_{H}\big{)}$ is the set of isolated vertices of the direct product $G\times H$ with $|I|=|I_{G}||V(H)|+|I_{H}||V(G)|-|I_{G}||I_{H}|$. On the other hand, $\gamma_{oir2}(G\times H)=\gamma_{oir2}(G^{-}\times H^{-})+|I|$. So, we may suppose that both $G$ and $H$ have no isolated vertices.

\begin{theorem}\label{direct}
Let $G$ and $H$ be two graphs with no isolated vertices. Then,
$$\gamma_{oir2}(G\times H)\leq min\{\gamma_{oir2}(H)|V(G)|,\gamma_{oir2}(G)|V(H)|\}.$$
Moreover, this bound is sharp. 
\end{theorem}
\begin{proof}
Let $g$ be a $\gamma_{oir2}(H)$-function. We define $f:V(G)\times V(H)\rightarrow \mathbb{P}(\{1,2\})$ by $f(x,y)=g(y)$ for each $(x,y)\in V(G)\times V(H)$. Suppose that $(x,y)(x',y')\in E(G\times H)$ for some $(x,y),(x',y')\in V_{\emptyset}^f$. This shows that $yy'\in E(H)$ and that $g(y)=g(y')=\emptyset$. This is a contradiction. Therefore, $V_{\emptyset}^f$ is an independent set in $G\times H$. 

For any $(x,y)\in V(G)\times V(H)$ with weight $\emptyset$ under $f$, the equality $g(y)=\emptyset$ implies that $g\big{(}N_{H}(y)\big{)}=\{1,2\}$. Moreover, $N_{G}(x)\neq \emptyset$ since $G$ has no isolated vertex. So, $f\big{(}N_{G\times H}(x,y)\big{)}=f\big{(}N_{G}(x)\times N_{H}(y)\big{)}=g\big{(}N_{H}(y)\big{)}=\{1,2\}$. Therefore, $f$ is an OI$2$RD function of $G\times H$. We then have $\gamma_{oir2}(G\times H)\leq \omega(f)=\gamma_{oir2}(H)|V(G)|$. Interchanging the roles of $G$ and $H$ establishes the upper bound.

That the bound is sharp, can be seen as follows. We consider the graph $K_{m}\times K_{n}$ with $m,n\geq2$. By the structure, at most $\alpha(K_{m}\times K_{n})=\max\{m,n\}$ vertices of $K_{m}\times K_{n}$ are assigned $\emptyset$ under any $\gamma_{oir2}(G\times H)$-function. This shows that 
$$\gamma_{oir2}(K_{m}\times K_{n})\geq mn-\max\{m,n\}=\min\{m(n-1),n(m-1)\},$$
satisfying the equality in the upper bound. 
\end{proof}

Regarding the Cartesian product $G\Box H$, we observe that
$$\gamma_{oir2}(G\Box H)=\gamma_{oir2}(G^{-}\Box H^{-})+|I_{G}|\gamma_{oir2}(H)+|I_{H}|\gamma_{oir2}(G)-|I_G||I_H|.$$
Hence, in what follows, it suffices to assume that both $B$ and $H$ have no isolated vertices. 

\begin{theorem}\label{Cartesian}
Let $G$ and $H$ be two graphs with no isolated vertices. Then,
$$\gamma_{oir2}(G\Box H)\leq \alpha(G)\beta(H)+\beta(G)|V(H)|-min\{\beta(G),\beta(H)\}.$$
Furthermore, this bound is sharp.
\end{theorem}
\begin{proof}
Without loss of generality, we may assume that $\beta(G)\leq \beta(H)$. Let $I$ and $J$ denote an $\alpha(G)$-set and an $\alpha(H)$-set, respectively. Suppose that $V(G)\setminus I=\{g_{1},\cdots,g_{r}\}$ and $V(H)\setminus J=\{h_{1},\cdots,h_{s}\}$, where $r=\beta(G)\leq \beta(H)=s$ by our assumption and the well-known Gallai theorem \cite{Gallai} (which states that $\alpha(F)+\beta(F)=|V(F)|$ for any graph $F$). It is a routine matter to see that 
$$S=\{(g_{i},h_{i})\mid 1\leq i\leq r\}\cup(I\times J)$$ 
is independent in $G\Box H$. Note that the subsets 
$$S,\ \ \ I\times(V(H)\setminus J)\ \ \ \mbox{and}\ \ \ \big{(}(V(G)\setminus I)\times V(H)\big{)}\setminus \{(g_{i},h_{i})\mid 1\leq i\leq r\}$$ 
form a partition of $V(G)\times V(H)$. We now define $h:V(G)\times V(H)\rightarrow \mathbb{P}(\{1,2\})$ by
\begin{equation*}
h((x,y))=\left \{
\begin{array}{lll}
\emptyset & \mbox{if}\ \ (x,y)\in S,\vspace{1.5mm}\\
\{1\} & \mbox{if}\ \ (x,y)\in I\times(V(H)\setminus J),\vspace{1.5mm}\\
\{2\} & \mbox{if}\ \ (x,y)\in\big{(}(V(G)\setminus I)\times V(H)\big{)}\setminus \{(g_{i},h_{i})\mid 1\leq i\leq r\}.
\end{array}
\right.
\end{equation*}

Notice that $S=V_{\emptyset}^{h}$ is independent in $G\Box H$ as mentioned above. Now let $(x,y)\in S$. We distinguish two possibilities depending on membership of $(x,y)$.\vspace{1mm}

($i$) Suppose that $(x,y)\in I\times J$. Since $H$ has no isolated vertices, $y$ is adjacent to a vertex $y'\in V(H)\setminus J$. So, $(x,y)$ is adjacent to $(x,y')$ with $h\big{(}(x,y')\big{)}=\{1\}$. Similarly, since $G$ has no isolated vertices, $x$ is adjacent to a vertex $x'\in V(G)\setminus I$. Moreover, $y\notin \{h_{1},\cdots,h_{s}\}$. This shows that $(x,y)$ is adjacent to $(x',y)$ with $h\big{(}(x',y)\big{)}=\{2\}$. Therefore, $h\big{(}N_{G\Box H}((x,y))\big{)}=\{1,2\}$.\vspace{1mm}

($ii$) Suppose that $(x,y)=(g_{i},h_{i})$ for some $1\leq i\leq r$. Since $I$ is an $\alpha(G)$-set, $x$ has a neighbor $x'\in I$. Therefore, $(x,y)$ is adjacent to $(x',y)\in I\times(V(H)\setminus J)$ for which $h\big{(}(x',y)\big{)}=\{1\}$. Because $H$ has no isolated vertices, there exists a vertex $y'\in V(H)$ adjacent to $y$. So, $(x,y)$ is adjacent to $(x,y')$. This implies that $(x,y')\notin \{(g_{i},h_{i})\mid 1\leq i\leq r\}$ as this set is independent. This shows that, $(x,y)$ is adjacent to $(x,y')\in\big{(}(V(G)\setminus I)\times V(H)\big{)}\setminus \{(g_{i},h_{i})\mid 1\leq i\leq r\}$ with $h\big{(}(x,y')\big{)}=\{2\}$. Therefore, $h\big{(}N_{G\Box H}((x,y))\big{)}=\{1,2\}$.\vspace{1mm}

The above discussion guarantees that $h$ is an OI$2$RD function of $G\Box H$. Thus, 
\begin{equation}\label{INE1}
\gamma_{oir2}(G\Box H)\leq \omega(h)=\alpha(G)\big{(}|V(H)|-\alpha(H)\big{)}+\big{(}|V(G)|-\alpha(G)\big{)}|V(H)|-\beta(G).
\end{equation}
On the other hand, we have $\alpha(F)+\beta(F)=|V(F)|$ for any graph $F$. Taking this fact into account, the desired upper bound follows from (\ref{INE1}).

That the upper bound is sharp, may be seen by considering $G=P_{m}$ and $H=K_{n}$ for $m,n\geq2$ and $n\geq \lfloor m/2\rfloor+1$. It is easily observed that $\gamma_{oir2}(P_{m}\Box K_{n})=m(n-1)$. Moreover, $m(n-1)=\alpha(P_{m})\beta(K_{n})+\beta(P_{m})|V(K_{n})|-\beta(P_{m})$. This completes the proof.
\end{proof}


\subsection{Rooted and corona products}

A \emph{rooted graph} is a graph in which one vertex is labeled in a special way to distinguish it from the other vertices. The special vertex is called the \emph{root} of the graph. Let $G$ be a labeled graph on $n$ vertices. Let ${\cal H}$ be a sequence of $n$ rooted graphs $H_1,\ldots,H_n$. The \emph{rooted product graph} $G( {\cal H})$ is the graph obtained by identifying the root of $H_i$ with the $i$th vertex of $G$ (see \cite{rooted-first}). We here consider the particular case of rooted product graphs where ${\cal H}$ consists of $n$  isomorphic rooted graphs \cite{Schwenk}. More formally, assuming that $V(G) = \{g_1,\ldots,g_n\}$ and that the root vertex of $H$ is $v$, we define the rooted product graph $G\circ_{v} H=(V,E)$, where $V=V(G)\times V(H)$ and
$$E=\displaystyle\bigcup_{i=1 }^n\{(g_i,h)(g_i,h')\mid hh'\in E(H)\}\cup \{(g_i,v)(g_j,v)\mid g_ig_j\in E(G)\}.$$

Note that subgraphs induced by $H$-layers of $G\circ_v H$ are isomorphic to $H$. We next study the OI$2$RD number of rooted product graphs.

\begin{theorem}\label{Rooted} 
Let $G$ be any graph of order $n$. If $H$ is any graph with root $v$, then
$$\gamma_{oir2}(G\circ_v H)\in \{n\gamma_{oir2}(H)-\alpha(G),n\gamma_{oir2}(H),n\gamma_{oir2}(H)+\beta(G)\}.$$
\end{theorem}
\begin{proof}
We first prove that 
\begin{equation}\label{lowerupper}
n\gamma_{oir2}(H)-\alpha(G)\leq \gamma_{oir2}(G\circ_v H)\leq n\gamma_{oir2}(H)+\beta(G).
\end{equation}
In order to prove the lower bound, let $f$ be a $\gamma_{oir2}(G\circ_v H)$-function. If $f_{x}=f\mid_{(G\circ_v H)[\{x\}\times V(H)]}$ is an OI$2$RD function of $H_{x}=(G\circ_v H)[\{x\}\times V(H)]\cong H$ for every $x\in V(G)$, then $\gamma_{oir2}(H)\leq \omega(f_{x})$ for all $x\in V(G)$. Therefore, 
$$\gamma_{oir2}(G\circ_v H)=\omega(f)=\sum_{x\in V(G)}\omega(f_{x})\geq n\gamma_{oir2}(H)>n\gamma_{oir2}(H)-\alpha(G).$$ 
So, in what follows, we may assume that $f_{x}$ is not an OI$2$RD function of $H_{x}$ for some $x\in V(G)$. Since $f$ is an OI$2$RD function of $G\circ_v H$, there are no two adjacent vertices of $H_{x}$ which are assigned $\emptyset$ under $f_{x}$. Therefore, there exists a vertex $(x,y)$ of $H_{x}$ with $f_{x}\big{(}(x,y)\big{)}=\emptyset$ for which $f_{x}(N_{H_{x}}\big{(}(x,y))\big{)}\neq\{1,2\}$. By the structure of rooted products and since $f$ is an OI$2$RD function of $G\circ_v H$, it follows that $y=v$. We now consider two cases.

\textit{Case $1$.} Let $(x,v)$ be an isolated vertex of $H_{x}$. In such a situation, $G\circ_v H$ is isomorphic to the disjoint union of one copy of $G$ and $n$ copies of the graph $H-v$. On the other hand,
\begin{equation*}
\begin{array}{lcl} 
\omega(f\mid_{V(G)})=\gamma_{oir2}(G)&=&|V_{\{1\}}\cap V(G)|+|V_{\{2\}}\cap V(G)|+2|V_{\{1,2\}}\cap V(G)|\\
&\geq& n-|V_{\emptyset}\cap V(G)|\geq n-\alpha(G)=\beta(G).
\end{array}
\end{equation*}
Therefore, $\gamma_{oir2}(G\circ_v H)=\gamma_{oir2}(G)+n\gamma_{oir2}(H-v)=\gamma_{oir2}(G)+n(\gamma_{oir2}(H)-1)\geq \beta(G)+n(\gamma_{oir2}(H)-1)=n\gamma_{oir2}(H)-\alpha(G)$.

\textit{Case $2$.} Suppose that $(x,v)$ is not an isolated vertex of $H_{x}$. Since $f_{x}((x,v))=\emptyset$, it follows that $(x,v)$ is adjacent to a vertex $(x,w)$ of $H_{x}$ for which $f_{x}\big{(}(x,w)\big{)}=\{1\}$ or $\{2\}$. Now the assignment $f_{x}'\big{(}(x,v)\big{)}=\{1\}$ and $f_{x}'\big{(}(x,u)\big{)}=f_{x}\big{(}(x,u)\big{)}$ for the other vertices $u\in V(H)$ defines an OI$2$RD function of $H_{x}$ with weight $\omega(f_{x})+1$. Therefore, $\gamma_{oir2}(H)\leq \omega(f_{x})+1$ for each vertex $x\in V(G)$ for which $f_{x}$ is not an OI$2$RD function of $H_{x}$. 

Now let $S$ be a $\beta(G)$-set. This shows that at least $\beta(G)$ vertices in $V(G)\times \{v\}$ are assigned at least $\{1\}$ or $\{2\}$ under $f$. Moreover, $f_{x}$ is an OI$2$RD function of $H_{x}$ for each vertex $x$ with $f\big{(}(x,v)\big{)}\neq \emptyset$. Therefore,
\begin{equation*}
\begin{array}{lcl}
\gamma_{oir2}(G\circ_v H)&=&\sum_{x\in V(G)}\omega(f_{x})=\sum_{x\in V(G)\setminus S}\omega(f_{x})+\sum_{x\in S}\omega(f_{x})\\
&\geq&\big{(}n-\beta(G)\big{)}(\gamma_{oir2}(H)-1)+\beta(G)\gamma_{oir2}(H)=n\gamma_{oir2}(H)-\alpha(G). 
\end{array}
\end{equation*}

We now prove the upper bound. Suppose that there exists a $\gamma_{oir2}(H)$-function $f$ for which $f(v)\neq \emptyset$. Clearly, $f$ results in a $\gamma_{oir2}(H_{x})$-function $f_{x}$ for each $x\in V(G)$. Therefore, $\gamma_{oir2}(G\circ_v H)\leq \sum_{x\in V(G)}\omega(f_{x})=n\gamma_{oir2}(H)\leq n\gamma_{oir2}(H)+\beta(G)$. Assume now that every $\gamma_{oir2}(H)$-function $f$ assigns $\emptyset$ to $v$. Since the vertices with weight $\emptyset$ under $f$ are independent, it follows that one $1$ or one $2$, say one $2$, belongs to $f\big{(}N_{H}(v)\big{)}$. Now let $S$ be a $\beta(G)$-set. We define $g:V(G)\times V(H)\rightarrow \mathbb{P}(\{1,2\})$ by
\begin{equation*}
g\big{(}(x,y)\big{)}=\left \{
\begin{array}{lll}
f(y) & \mbox{if}\ \ y\neq v,\vspace{1.5mm}\\
\{1\} & \mbox{if}\ \ y=v\ \mbox{and}\ x\in S,\vspace{1.5mm}\\
\emptyset & \mbox{if}\ \ y=v\ \mbox{and}\ x\in V(G)\setminus S.
\end{array}
\right.
\end{equation*}
It is then easy to check that $g$ is an OI$2$RD function of $G\circ_v H$ with weight $\omega(g)=n\gamma_{oir2}(H)+\beta(G)$, implying the upper bound.

Note that if $G$ is edgeless, then $\gamma_{oir2}(G\circ_v H)=n\gamma_{oir2}(H)=n\gamma_{oir2}(H)+\beta(G)$. So, in what follows we assume that $G$ is not edgeless. We distinguish the following cases depending on the behavior of $\gamma_{oir2}(H)$-functions.

\textit{Case $3$.} Let every $\gamma_{oir2}(H)$-function assign $\emptyset$ to $v$. Let $f$ be a $\gamma_{oir2}(G\circ_v H)$-function. Suppose to the contrary that there exists a vertex $x\in V(G)$ for which $\omega(f_{x})\leq \gamma_{oir2}(H)-1$, where $f_{x}=f|_{\{x\}\times V(H)}$. By the properties of the rooted product graph $G\circ_v H$ and since $f_{x}$ is not an OI$2$RD function of $H_{x}$, we have $f(x,v)=f_{x}\big{(}(x,v)\big{)}=\emptyset$ and $|f\big{(}N_{H_{x}}(x,v)\big{)}|=1$. Now $g(y)=f_{x}(x,y)$ for $y\in V(H)\setminus\{v\}$, and $g(v)=\{1\}$ defines an OI$2$RD function of $H$ with $\omega(g)=\gamma_{oir2}(H)$ for which $g(v)\neq \emptyset$. This is a contradiction. Therefore, $\omega(f_{x})\geq \gamma_{oir2}(H)$ for each $x\in V(G)$. Set $A=\{x\in V(G)\mid w(f_{x})=\gamma_{oir2}(H)\}$ and $B=\{x\in V(G)\mid w(f_{x})>\gamma_{oir2}(H)\}$. Obviously, $|A|+|B|=n$. We then have
\begin{equation} \label{root5}
\omega(f)\geq|A|\gamma_{oir2}(H)+|B|(\gamma_{oir2}(H)+1)=n\gamma_{oir2}(H)+|B|.
\end{equation}
If $f_{x}\big{(}(x,v)\big{)}\neq \emptyset$ for some vertex $x\in A$, then it is easy to see that $g(y)=f_{x}\big{(}(x,y)\big{)}$ for all $y\in V(G)$ is an OI$2$RD function of $H$ with weight $\gamma_{oir2}(H)$ for which $g(v)\neq \emptyset$. This is a contradiction. Therefore, $f_{x}\big{(}(x,v)\big{)}=\emptyset$ for all $x\in A$. This implies that $\{(x,v)\mid x\in A\}$ is independent in $(G\circ_v H)[V(G)\times \{v\}]\cong G$. Hence $|A|\leq \alpha(G)$, implying that $|B|=n-|A|\geq \beta(G)$. This results in $\gamma_{oir2}(G\circ_v H)=\omega(f)\geq n\gamma_{oir2}(H)+\beta(G)$ by (\ref{root5}). Therefore, $\gamma_{oir2}(G\circ_v H)=n\gamma_{oir2}(H)+\beta(G)$ by (\ref{lowerupper}). 

\textit{Case $4$.} Suppose that $g(v)\neq \emptyset$ for some $\gamma_{oir2}(H)$-function $g$. We need to consider two subcases depending on the collection of such functions $g$.

\textit{Subcase $4$.$1$.} Let there exist such a function $g$ under which $v$ is not adjacent to any vertex of $H$ with weight $\emptyset$. Notice that at least one $1$ or one $2$, say one $1$, belongs to $g\big{(}N_{H}(v)\big{)}$ in $H$. Let $I$ be an $\alpha(G)$-set. Then, $h:V(G)\times V(H)\rightarrow \mathbb{P}(\{1,2\})$ defined by
\begin{equation*}
h\big{(}(x,y)\big{)}=\left \{
\begin{array}{lll}
g(y) & \mbox{if}\ \ x\in V(G)\ \mbox{and}\ y\neq v,\vspace{1.5mm}\\
\emptyset & \mbox{if}\ \ x\in I\ \mbox{and}\ y=v,\vspace{1.5mm}\\
\{2\} & \mbox{if}\ \ x\in V(G)\setminus I\ \mbox{and}\ y=v,
\end{array}
\right.
\end{equation*}
is an OI$2$RD function of $G\circ_v H$ with weight $n\gamma_{oir2}(H)-\alpha(G)$. This implies that $\gamma_{oir2}(G\circ_v H)=n\gamma_{oir2}(H)-\alpha(G)$ in view of (\ref{lowerupper}).

\textit{Subcase $4$.$2$.} Suppose now that for all such functions $g$, $v$ is adjacent to a vertex of $H$ with weight $\emptyset$ under $g$. Note that $h\big{(}(x,y)\big{)}=g(y)$ for all $x\in V(G)$ and $y\in V(H)$ defines an OI$2$RD function of $G\circ_v H$ with weight $n\gamma_{oir2}(H)$. So, $\gamma_{oir2}(G\circ_v H)\leq n\gamma_{oir2}(H)$. We again suppose that $f$ is a $\gamma_{oir2}(G\circ_v H)$-function. Let $\omega(f_{x})\leq \gamma_{oir2}(H)-1$ for some $x\in V(G)$. This shows that $f_{x}\big{(}(x,v)\big{)}=\emptyset$, for otherwise $f_{x}$ would be an OI$2$RD function of $H_{x}\cong H$ with a weight less than $\gamma_{oir2}(H)$, which is impossible. Therefore, all neighbors of $(x,v)$ in $\{x\}\times V(H)$ must have nonempty weights under $f$. But the assignment $h(v)=\{1\}$, and $h(y)=f\big{(}(x,y)\big{)}$ for other vertices defines an OI$2$RD function of $H$ with weight $\gamma_{oir2}(H)$ (that is, a $\gamma_{oir2}(H)$-function) for which no vertices adjacent to $v$ are assigned $\emptyset$ under $h$. This contradicts the fact that all $\gamma_{oir2}(H)$-functions assigning a nonempty weight to $v$ assign $\emptyset$ to a neighbor of it. The above argument guarantees that $\omega(f_{x})\geq \gamma_{oir2}(H)$ for all $x\in V(G)$. Therefore, $\gamma_{oir2}(G\circ_v H)=\omega(f)\geq n\gamma_{oir2}(H)$. This results in $\gamma_{oir2}(G\circ_v H)=n\gamma_{oir2}(H)$.

All in all, we have shown that $\gamma_{oir2}(G\circ_v H)$ belongs to $\{n\gamma_{oir2}(H)-\alpha(G),n\gamma_{oir2}(H),n\gamma_{oir2}(H)+\beta(G)\}$. This completes the proof.
\end{proof}

Let $G$ and $H$ be graphs where $V(G)=\{v_1,\ldots,v_{n}\}$. We recall that the corona $G\odot H$ of graphs $G$ and $H$ is obtained from the disjoint union of $G$ and $n$ disjoint copies of $H$, say $H_1,\ldots,H_{n}$, such that for all $i\in \{1,\dots,n\}$, the vertex $v_i\in V(G)$ is adjacent to every vertex of $H_i$.

Unlike the cases of Cartesian and direct products, the existence of isolated vertices in $H$ is irrelevant to the number of components of $G\odot H$. In particular, if $H$ has isolated vertices, $G\odot H$ remains connected when $G$ is connected. In fact, as we next show, the exact formula for $\gamma_{oir2}(G\odot H)$ changes in the case when $H$ has isolated vertices. In particular, when $|V(H)|=1$, it establishes the NP-hardness of the problem of computing $\gamma_{oir2}$ even for some special families of graphs (see \cite{gj} and \cite{mm}). 

Cabrera Mart\'{i}nez \cite{acm} proved that $\gamma_{oir2}(G\odot H)=|V(G)|(|V(H)|+1)-|V(G)|\alpha(H)$ for all graphs $G$ and $H$ with no isolated vertices. In what follows, we present an exact formula for $\gamma_{oir2}(G\odot H)$ for any graph $G$ with no isolated vertices and arbitrary graph $H$. By the way, the method by which we prove the following theorem is different from that of \cite{acm}.

\begin{theorem}
Let $G$ be a graph of order $n$ with no isolated vertices and let $H$ be any graph with $i_{H}$ isolated vertices. If $|V(H)|=1$, then 
$$\gamma_{oir2}(G\odot H)=n+\beta(G).$$ 
If $|V(H)|\geq2$, then
$$\gamma_{oir2}(G\odot H)=\left\{
\begin{array}
[c]{ccc}%
n(\beta(H)+1) & \text{if }\ i_{H}=0,\\
n(\beta(H)+2) & \text{if }\ i_{H}\neq0.
\end{array}
\right.$$
\end{theorem}
\begin{proof}
We first suppose that $H\cong K_{1}$. Let $V(G)=\{v_{1},\cdots,v_{n}\}$. Then, $G\odot K_{1}$ is obtained from $G$ by joining $n$ new vertices $u_{1},\cdots,u_{n}$ to $v_{1},\cdots,v_{n}$, respectively. Let $f$ be a $\gamma_{oir2}(G\odot K_{1})$-function. Clearly, $1\leq|f(v_{i})|+|f(u_{i})|\leq2$ for each $1\leq i\leq n$. If $|f(v_{i})|+|f(u_{i})|=2$ for some $1\leq i\leq n$, we may assume that $f(v_{i})=\{1,2\}$ and $f(u_{i})=\emptyset$. Moreover, $f(u_{i})=\{1\}$ or $\{2\}$ whenever $|f(v_{i})|+|f(u_{i})|=1$. Let $A=\{1\leq i\leq n\mid |f(v_{i})|+|f(u_{i})|=1\}$. Note that 
\begin{equation}\label{Tak}
\gamma_{oir2}(G\odot K_{1})=\sum_{i\notin A}(|f(v_{i})|+|f(u_{i})|)+\sum_{i\in A}(|f(v_{i})|+|f(u_{i})|)=2n-|A|.
\end{equation}

On the other hand, $|A|\leq \alpha(G)$ as the vertices $v_{i}$, for which $i\in A$, are assigned $\emptyset$ under $f$. So, $\gamma_{oir2}(G\odot K_{1})\geq 2n-\alpha(G)=n+\beta(G)$ by (\ref{Tak}). 

Let $I$ be a $\alpha(G)$-set. We can observe that the assignment $\emptyset$ to the vertices in $I$, $\{1\}$ to the other vertices of $G$, and $\{2\}$ to $u_{1},\cdots,u_{n}$ defines an OI$2$RD function of $G\odot K_{1}$ with weight $n+\beta(G)$. Therefore, $\gamma_{oir2}(G\odot K_{1})\leq n+\beta(G)$. This results in the equality for the case when $|V(H)|=1$.

Now let $|V(H)|\geq2$. Consider a function $f_{G}=(V_{\emptyset},V_{\{1\}},V_{\{2\}},V_{\{1,2\}})$ of $G$ such that $V_{\emptyset}$ is an independent set. We next define a function $f:V(G\odot H)\rightarrow \mathbb{P}(\{1,2\})$ as follows. Suppose that $v_{i}\in V(G)=\{v_{1},\cdots,v_{n}\}$.

($a$) If $f_{G}(v_{i})=\emptyset$, we consider a function $f_{H}$ which assigns $\{1\}$ and $\{2\}$ to the vertices of $H_{i}$ for which both $\{1\}$ and $\{2\}$ appear at least one time to the vertices of $H_{i}$, and let $f(w)=f_{H}(w)$ for all $w\in V(H_{i})$.

($b$) Let $f_{G}(v_{i})=\{1\}$ (resp. $f_{G}(v_{i})=\{2\}$). Suppose that $I$ is an $\alpha(H)$-set and $I_{H}$ is the set of isolated vertices of $H$. Clearly, $I_{H}\subseteq I$. We define $g_{H}$ by $g_{H}(w)=\emptyset$ for each $w\in I\setminus I_{H}$, and $g_{H}(w)=\{2\}$ (resp. $g_{H}(w)=\{1\}$) for the other vertices $w$ of $H$. We next let $f(w)=g_{H}(w)$ for each $w\in V(H_{i})$.

($c$) Let $f_{G}(v_{i})=\{1,2\}$. Define $k_{H}$ by $k_{H}(w)=\emptyset$ for each $w\in I$, and $k_{H}(w)=\{1\}$ for the other vertices. We next let $f(w)=k_{H}(w)$ for each $w\in V(H_{i})$.

($d$) For each $1\leq i\leq n$, let $f(v_{i})=f_{G}(v_{i})$.

It is not hard to see that the above mentioned function $f$ is an OI$2$RD function of $G\odot H$ with weight
\begin{equation*}
\begin{array}{lcl} 
\omega(f)=|V_{\emptyset}||V(H)|&+&|V_{\{1\}}|(|V(H)|-\alpha(H)+i_{H}+1)\\
&+&|V_{\{2\}}|(|V(H)|-\alpha(H)+i_{H}+1)+|V_{\{1,2\}}|(|V(H)|-\alpha(H)+2). 
\end{array}
\end{equation*}
Since $f_{G}=(V_{\emptyset},V_{\{1\}},V_{\{2\}},V_{\{1,2\}})$ is an arbitrary function of $G$ for which $V_{\emptyset}$ is independent and because $\alpha(H)+\beta(H)=|V(H)|$, we deduce that
\begin{align*}
  \gamma_{oir2}(G\odot H) & \leq min\{|V_{\emptyset}||V(H)|+|V_{\{1\}}|(\beta(H)+i_{H}+1)\\
   &\hspace*{0.0cm}+|V_{\{2\}}|(\beta(H)+i_{H}+1)+|V_{\{1,2\}}|(\beta(H)+2)\},
\end{align*}
taken over all possible function $f_{G}=(V_{\emptyset},V_{\{1\}},V_{\{2\}},V_{\{1,2\}})$ of $G$ for which $V_{\emptyset}$ is an independent set in $G$.

On the other hand, let $g=(V'_{\emptyset},V'_{\{1\}},V'_{\{2\}},V'_{\{1,2\}})$ be a $\gamma_{oir2}(G\odot H)$-function and let $v_i\in V(G)$. We consider the following cases.

\textit{Case 1.} $g(v_{i})=\emptyset$. Since $V'_{\emptyset}$ is independent, we have $|g(w)|\geq1$ for all $w\in V(H_{i})$. Therefore, $g(V(H_{i})\cup\{v_{i}\})\geq|V(H)|$.

\textit{Case 2.} $g(v_{i})=\{1\}$ or $\{2\}$. Since $V'_{\emptyset}$ is independent, at most $\alpha(H)$ vertices of $H_{i}$ can be assigned $\emptyset$ under $g$. Moreover, the isolated vertices of $H_{i}\cong H$ cannot be assigned $\emptyset$ under $g$. Therefore, $|g(w)|\geq1$ for all $w\in I\setminus I_{H}$. This implies that $g(V(H_{i})\cup\{v_{i}\})\geq|V(H)|-\alpha(H)+i_{H}+1=\beta(H)+i_{H}+1$.

\textit{Case 3.} $g(v_{i})=\{1,2\}$. Note that at most $\alpha(H)=|V(H)|-\beta(H)$ vertices of $H_{i}$ can be assigned $\emptyset$ under $g$ by a similar fashion. Therefore, $g(V(H_{i})\cup\{v_{i}\})\geq \beta(H)+2$.

On the other hand, since $g$ is an OI$2$RD function of $G\odot H$, it follows that the function $f''_{G}=(V''_{\emptyset},V''_{\{1\}},V''_{\{2\}},V''_{\{1,2\}})=(V'_{\emptyset}\cap V(G),V'_{\{1\}}\cap V(G),V'_{\{2\}}\cap V(G),V'_{\{1,2\}}\cap V(G))$ fulfills the independence of $V''_{\emptyset}=V'_{\emptyset}\cap V(G)$. As a consequence of all the cases above, we deduce that
\begin{align*}
\gamma_{oir2}(G\odot H) & =\sum_{i=1}^n g(V(H_i)\cup\{v_i\})\geq |V''_{\emptyset}||V(H)|+|V''_{\{1\}}|(\beta(H)+i_{H}+1)\\
   &\hspace*{0.0cm}+|V''_{\{2\}}|(\beta(G)+i_{H}+1)+|V''_{\{1,2\}}|(\beta(H)+2)\\
   &\hspace*{0.0cm}\geq min\{|V_{\emptyset}||V(H)|+|V_{\{1\}}|(\beta(H)+i_{H}+1)\\
   &\hspace*{0.0cm}+|V_{\{2\}}|(\beta(H)+i_{H}+1)+|V_{\{1,2\}}|(\beta(H)+2)\},
\end{align*}
taken over all possible functions $f_{G}=(V_{\emptyset},V_{\{1\}},V_{\{2\}},V_{\{1,2\}})$ of $G$ for which $V_{\emptyset}$ is independent in $G$. Therefore,
\begin{align*}
  \gamma_{oir2}(G\odot H) & =min\{|V_{\emptyset}||V(H)|+|V_{\{1\}}|(\beta(H)+i_{H}+1)\\
   &\hspace*{0.0cm}+|V_{\{2\}}|(\beta(H)+i_{H}+1)+|V_{\{1,2\}}|(\beta(H)+2)\},
\end{align*}
taken over all possible above-mentioned functions $f_{G}$.

Clearly, $\beta(H)+1\leq min\{\beta(H)+2,|V(H)|\}$ when $H$ has no isolated vertices, and $\beta(H)+2\leq min\{\beta(H)+i_{H}+1,|V(H)|\}$ otherwise. Taking these facts into consideration, the function $f_{G}$ for which we get the minimum in the right-hand side of the last equality assigns $\{1\}$ or $\{2\}$ to all vertices of $G$ when $H$ has no isolated vertices, and assigns $\{1,2\}$ to all vertices of $G$ otherwise. Consequently, $\gamma_{oir2}(G\odot H)=n(\beta(H)+1)$ when $i_{H}=0$, and $\gamma_{oir2}(G\odot H)=n(\beta(H)+2)$ when $i_{H}\neq0$. This completes the proof.
\end{proof}


\end{document}